\documentclass[reqno,12pt]{amsart}

\pdfoutput=1

\usepackage{color} % black,white,red,green,blue,cyan,magenta,yellow
\usepackage{ifpdf}
\ifpdf
    \usepackage[pdftex]{graphicx}
    \usepackage{pdftricks}
    \begin{psinputs}
      \usepackage{pstricks}
    \end{psinputs}
    \usepackage[pdftex]{hyperref}
    \hypersetup{
        unicode=false,          % non-Latin characters in Acrobat’s bookmarks
        pdftoolbar=true,        % show Acrobat’s toolbar?
        pdfmenubar=true,        % show Acrobat’s menu?
        pdffitwindow=false,     % window fit to page when opened
        pdfstartview={FitH},    % fits the width of the page to the window
        pdftitle={MCP Article},      % title
        pdfauthor={Michael Holst},   % author
        pdfsubject={Mathematics},    % subject of the document
        pdfcreator={Michael Holst},  % creator of the document
        pdfproducer={Michael Holst}, % producer of the document
        pdfkeywords={PDE, analysis, mathematical physics}, % list of keywords
        pdfnewwindow=true,      % links in new window
        colorlinks=true,        % false: boxed links; true: colored links
        linkcolor=red,          % color of internal links
        citecolor=blue,         % color of links to bibliography
        filecolor=magenta,      % color of file links
        urlcolor=cyan           % color of external links
    }

\else
    \usepackage{graphicx}
    \usepackage{pstricks}
    \newenvironment{pdfpic}{}{}

\fi

\usepackage{times}
\usepackage{amsfonts}

\usepackage{amsmath}
\usepackage{amsthm}
\usepackage{amssymb}
\usepackage{amsbsy}
\usepackage{amscd}

\usepackage{enumerate}
\usepackage{verbatim}
\usepackage{subfigure}

\newtheorem{theorem}{Theorem}[section]

\newtheorem{lemma}[theorem]{Lemma}

\numberwithin{equation}{section}  %amsmath command: tie counter to section

  \newcounter{mnote}
  \let\oldmarginpar\marginpar
    \renewcommand\marginpar[1]{\-\oldmarginpar[\raggedleft\footnotesize #1]%
    {\raggedright\footnotesize #1}}
\definecolor{myblue}{rgb}{0.2,0.2,0.7}
\definecolor{mygreen}{rgb}{0,0.6,0}
\definecolor{mycyan}{rgb}{0,0.6,0.6}
\definecolor{myred}{rgb}{0.9,0.2,0.2}
\definecolor{mymagenta}{rgb}{0.9,0.2,0.9}
\definecolor{mywhite}{rgb}{1.0,1.0,1.0}
\definecolor{myblack}{rgb}{0.0,0.0,0.0}

\newcommand{\beq}{\begin{equation}}
\newcommand{\eeq}{\end{equation}}
\newcommand{\beqa}{\begin{eqnarray}}
\newcommand{\eeqa}{\end{eqnarray}}
\begin{document}

\title[Finite Element Approximation of the Nonlinear Poisson-Boltzmann Equation]
      {The Finite Element Approximation of the \\
       Nonlinear Poisson--Boltzmann Equation}

\author[L. Chen]{Long Chen}
\email{clong@ucsd.edu}
\address{Department of Mathematics, University of California at San Diego, La Jolla, CA 92093.}
\thanks{The first author was supported in part by NSF awards 0411723 and 022560, in part by DOE awards DE-FG02-04ER25620 and DE-FG02-05ER25707, and in part by NIH award P41RR08605.} 

\author[M. Holst]{Michael Holst}
\email{mholst@math.ucsd.edu}
\address{Department of Mathematics, University of California at San Diego, La Jolla, CA 92093}
\thanks{The second author was supported in part by NSF awards 0411723 and 022560, in part by DOE awards DE-FG02-04ER25620 and DE-FG02-05ER25707, and in part by NIH award P41RR08605.}

\author[J. Xu]{Jinchao Xu}
\email{xu@math.psu.edu}
\address{The School of Mathematical Science, Peking University, Beijing, 100871 China and Department of Mathematics, Pennsylvania State University, University Park, PA 16801}
\thanks{The third author was supported in part by NSF DMS0308946, DMS-0619587, DMS-0609727, and NSFC-10528102.}

%\date{\today}
\date{November 20, 2006}


\begin{abstract}
A widely used electrostatics model in the biomolecular modeling 
community, the nonlinear Poisson--Boltzmann equation, along with its 
finite element approximation, are analyzed in this paper. A regularized 
Poisson--Boltzmann equation is introduced as an auxiliary problem, making 
it possible to study the original nonlinear equation with delta distribution 
sources. A priori error estimates for the finite element approximation
are obtained for the regularized Poisson--Boltzmann equation based on certain
quasi-uniform grids in two and three dimensions. Adaptive finite element 
approximation through local refinement driven by an a posteriori error
estimate is shown to converge. The Poisson--Boltzmann equation does not 
appear to have been previously studied in detail theoretically, and it is 
hoped that this paper will help provide molecular modelers with a better 
foundation for their analytical and computational work with the 
Poisson--Boltzmann equation. Note that this article apparently gives 
the first rigorous convergence result for a numerical discretization 
technique for the nonlinear Poisson--Boltzmann equation with delta distribution 
sources, and it also introduces the first provably convergent adaptive method 
for the equation.  This last result is currently one of only a handful of 
existing convergence results of this type for nonlinear problems.
\end{abstract}

\maketitle

%\clearpage

%\vspace*{-0.5cm}
{\footnotesize
\tableofcontents
}
\vspace*{-0.5cm}

%sec 1
\section{Introduction}\label{sec1}
In this paper, we shall design and analyze finite element approximations 
of a widely used electrostatics model in the biomolecular modeling community, 
the nonlinear Poisson--Boltzmann equation (PBE):
\begin{equation}\label{basicPBE}
-\nabla \cdot (\varepsilon \nabla \tilde u) + \bar{\kappa} ^2\sinh
(\tilde u) = \sum _{i=1}^{N_m}q_i\delta _i \quad \hbox{in }\; \mathbb
R^{d}, \; d=2,3,
\end{equation}
where the dielectric $\varepsilon$ and the modified Debye--H\"uckel
parameter $\bar{\kappa}$ are piecewise constants in domains $\Omega_m$ 
(the domain for the biomolecule of interest) and $\Omega _s$ (the 
domain for a solvent surrounding the biomolecule), and 
$\delta_i:=\delta (x-x_i)$ is a Dirac distribution at point $x_i$. The
importance of (\ref{basicPBE}) in biomolecular modeling is
well-established; cf.~\cite{Briggs.J;McCammon.J1990,Sharp.K;Honig.B1990} 
for thorough discussions. Some analytical solutions are known, but only 
for unrealistic structure geometries, and usually only for linearizations 
of the equation; cf.~\cite{Holst.M1994} for a collection of these 
solutions and for references to the large amount of literature on
analytical solutions to the PBE and similar equations. The current 
technological advances are more demanding and require the solution of 
highly nonlinear problems in complicated geometries. To this end, 
numerical methods, including the finite element method, are widely 
used to solve the nonlinear PBE 
\cite{Holst.M1994,Holst.M;Said.F1995,Baker.N;Holst.M;Wang.F2000,Baker.N;Sept.D;Holst.M;McCammon.J2001,Shestakov.A;Milovich.J;Noy.A2002,Chen.W;Shen.Y;Xia.Q2005,Yang.Y2005}.

The main difficulties for the rigorous analysis and provably good 
numerical approximation of solutions to the nonlinear Poisson--Boltzmann 
equation include: (1) Dirac distribution sources, (2) exponential rapid
nonlinearities, and (3) discontinuous coefficients. We shall address
these difficulties in this paper. To deal with the $\delta$ distribution 
sources, we decompose $\tilde u$ as an unknown function in $H^1$ and
a known singular function, namely,
$$
\tilde u = u + G, \quad \hbox{ with }\; G=\sum _{i=1}^{N_m}G_i,
$$
where $G_i$ is the fundamental solution of $-\varepsilon _m\Delta 
G_i=q_i\delta _i$ in $\mathbb R^{d}$. Substituting this decomposition 
into the PBE, we then obtain the so-called regularized Poisson--Boltzmann
equation (RPBE):
\begin{eqnarray*}
-\nabla \cdot (\varepsilon \nabla u) + \bar{\kappa} ^2\sinh (u+G) =
\nabla \cdot ((\varepsilon -\varepsilon _m)\nabla G) \quad \hbox{in
}\; \mathbb R^{d}, \; d=2,3.
\end{eqnarray*}
The singularities of the $\delta$ distributions are transferred to 
$G$, which then exhibits degenerate behavior at each $\{x_i\}\subset 
\Omega _m$.  At those points, both $\sinh G(x_i)$ and $\nabla G(x_i)$ 
exhibit blowup. However, since $G$ is known analytically, one avoids 
having to build numerical approximations to $G$.  Moreover, both of 
the coefficients $\bar{\kappa}$ and $\varepsilon-\varepsilon _m$ are 
zero inside $\Omega _m$ where the blowup behavior arises. Due to this
cutoff nature of coefficients, we obtain a well-defined nonlinear
second-order elliptic equation for the regularized solution $u$ with
a source term in $H^{-1}$.  We will show that it also admits a unique
solution $u\in H^1$, even though the original solution $\tilde u\notin 
H^1$ due to the singularities present in $G$.\enlargethispage{-10pt}

Singular function expansions are a common technique in applied and 
computational mathematics for this type of singularity; this type of 
expansion has been previously proposed for the Poisson--Boltzmann 
equation in~\cite{Zhou.Z;Payne.P;Vasquez.M1996} and was shown
(empirically) to allow for more accurate finite difference 
approximations. In their work, the motivation for the technique was 
the poor discrete approximation of arbitrarily placed delta distributions 
using only the fixed corners of uniform finite difference meshes.  In 
the present work, our interest is in developing finite element methods 
using completely unstructured meshes, so we are able to place the 
delta distributions precisely where they should be and do not have this
problem with approximate delta function placement. Our motivation 
here for considering a singular function expansion is rather that the 
solution to the Poisson--Boltzmann equation is simply not smooth enough 
to either analyze or approximate using standard methods without using 
some sort of two-scale or multiscale expansion that represents the
nonsmooth part of the solution analytically. In fact, it will turn
out that expanding the solution into the sum of three functions, 
namely, a known singular function, an unknown solution to a linear
auxiliary problem, and an unknown solution to a second nonlinear 
auxiliary problem, is the key to establishing some fundamental results 
and estimates for the continuous problem and is also the key to
developing a complete approximation theory for the discrete problem as 
well as provably convergent nonadaptive and adaptive numerical
methods.

Starting with some basic results on existence, uniqueness, and  a
priori estimates for the continuous problem, we analyze the finite
element discretization and derive discrete analogues of the
continuous results to show that discretization leads to a well-posed 
discrete problem. Using maximum principles for the continuous and 
discrete problems, we derive a priori $L^{\infty}$-estimates for
the continuous and discrete solutions to control the nonlinearity, 
allowing us to obtain a priori error estimates for our finite
element approximation of the form 
$$
 \|u-u_h\|_{1}\lesssim \inf _{v_h\in V^h_D}\|u-v_h\|_{1},
$$
where $V^h_D$ is the linear finite element subspace defined over 
quasi-uniform triangulations with a certain boundary condition, and
$u_h$ is the finite element approximation of $u$ in $V^h_D$. The 
result is {\em quasi-optimal} in the sense it implies that the finite 
element approximation to the RPBE is within a constant of being the 
best approximation from the subspace $V^h_D$. After establishing these 
results for finite element approximations, we describe an adaptive 
approximation algorithm that uses mesh adaptation through local 
refinement driven by a posteriori error estimates. The adaptive
algorithm can be viewed as a mechanism for dealing with the primary 
remaining difficulty in the RPBE, namely, the discontinuities of the
coefficients across the interface between the solvent and the 
molecular regions. Finally, we shall prove that our adaptive finite 
element method will produce a sequence of approximations that 
converges to the solution of the continuous nonlinear PBE. This last 
result is one of only a handful of existing results of this type for 
nonlinear elliptic equations (the others 
being~\cite{Dorfler.W1995,Veeser.A2002,Carstensen.C2006}).\enlargethispage{-10pt}

The outline of this paper is as follows. In section \ref{sec2}, we give 
a brief derivation and overview of the Poisson--Boltzmann equation. In 
section \ref{sec3}, we derive a regularized form of the Poisson--Boltzmann 
equation by using a singular function expansion. In section \ref{sec4}, 
we give some basic existence and uniqueness results for the
RPBE. In section \ref{sec5}, we derive an a priori $L^{\infty}$-estimate
for the continuous problem. After introducing finite element methods for the
RPBE, in section \ref{sec6} we derive an analogous a priori
$L^{\infty}$-estimate for the discrete problem, and based on this we 
obtain a quasi-optimal a priori error estimate for the finite
element approximation. In section \ref{sec7}, we describe the adaptive 
algorithm, present an a posteriori error estimate, and prove a
general convergence result for the algorithm. In the last section, we 
summarize our work and give further remarks on the practical aspects 
using results in the present paper.

%sec 2
\section{The Poisson--Boltzmann equation}\label{sec2}
In this section we shall give a brief introduction to the nonlinear 
Poisson--Boltzmann equation. A detailed derivation can be found in 
\cite{Tanford.C1961,Holst.M1994}.

\looseness=-1The nonlinear PBE, a second-order nonlinear partial differential
equation, is fundamental to Debye--H\"uckel continuum electrostatic
theory~\cite{P.Debye;E.Huckel1923}. It determines a dimensionless 
potential around a charged biological structure immersed in a salt 
solution. The PBE arises from the Gauss law, represented mathematically by
the Poisson equation, which relates the electrostatic potential $\Phi$ 
in a dielectric to the charge density $\rho$:
$$
 -\nabla \cdot(\varepsilon \nabla \Phi ) = \rho,
$$
where $\varepsilon$ is the dielectric constant of the medium and here is
typically piecewise constant. Usually it jumps by one or two orders of 
magnitude at the interface between the charged structure (a biological 
molecular or membrane) and the solvent (a salt solution). The charge 
density $\rho$ consist of two components: $\rho =\rho _{\rm macro} + 
\rho _{\rm ion}.$ For the macromolecule, the charge density is a 
summation of $\delta$ distributions at $N_m$ point charges in the point
charge behavior, i.e.,
$$
\rho _{\rm macro}(x)=\sum _{i=1}^{N_m}q_i\delta (x-x_i),
\quad \; q_i= \frac{4\pi e_c^2}{\kappa_B T}z_i,
$$
where $\kappa _B>0$ is the Boltzmann constant, $T$ is the temperature, 
$e_c$ is the unit of charge, and $z_i$ is the amount of charge.

For the mobile ions in the solvent, the charge density $\rho _{\rm ion}$ 
cannot be given in a deterministic way. Instead it will be given by 
the Boltzmann distribution. If the solvent contains $N$ types of ions, 
of valence $Z_i$ and of bulk concentration $c_i$, then a Boltzmann
assumption about the equilibrium distribution of the ions leads to
$$
 \rho _{{\rm ion}}=\sum _{i=1}^{N} c_iZ_ie_c \exp\left(-Z_i \frac{e_c
 \Phi}{\kappa _B T}\right).
$$
For a symmetric $1:1$ electrolyte, $N=2$, $c_i=c_0$, and $Z_i=(-1)^i$, 
which yields
\begin{eqnarray*}
\rho _{\rm ion}= -2c_0 e_c \sinh \left(\frac {e_c\Phi}{\kappa _B T}\right).
\end{eqnarray*}

We can now write the PBE for modeling the electrostatic potential
of a solvated biological structure. Let us denote the molecule region 
by $\Omega _m \subset \mathbb R^d$ and consider the solvent region 
$\Omega _s=\mathbb R^d\backslash \bar \Omega _m$. We use $\tilde u$ to 
denote the dimensionless potential and $\bar{\kappa} ^2$ to denote the 
modified Debye--H\"uckel parameter (which is a function of the ionic
strength of the solvent). The nonlinear Poisson--Boltzmann equation is 
then
\begin{align}\label{PBE}
-\nabla \cdot (\varepsilon \nabla \tilde u)+\bar{\kappa} ^2\sinh
(\tilde u)&=\sum _{i=1}^{N_m}q_i\delta _i \quad \hbox{in }\mathbb
R^d,\\
\tilde u(\infty)&=0,\label{eq2.2}
\end{align}
where
$$
 \varepsilon =\left\{
\begin{array}{rrl}
\varepsilon _m & \hbox{ if } & x\in \Omega _m,\\ \varepsilon _s &
\hbox{ if } & x\in \Omega _s,
\end{array}
\right. \quad \hbox{ and }\quad \bar{\kappa} =\left\{
\begin{array}{rrl}
0 & \hbox{ if } & x\in \Omega _m,\\ \sqrt{\varepsilon_s} \kappa >0&
\hbox{ if } & x\in \Omega _s.
\end{array}
\right.
$$
It has been determined empirically that $\varepsilon _m\approx 2$ 
and $\varepsilon _s \approx 80$. The structure itself (e.g., a biological 
molecule or a membrane) is represented implicitly by $\varepsilon$
and $\bar{\kappa}$, as well as explicitly by the $N_m$ point charges 
$q_i = z_i e_c$ at the positions $x_i$. The charge positions are 
located in the strict interior of the molecular region $\Omega_m$. A 
physically reasonable mathematical assumption is that all charge 
locations obey the following lower bound on their distance to the 
solvent region $\Omega_s$ for some $\sigma > 0$:
\begin{equation}
\label{sigma}
|x - x_i| \ge \sigma ~~~ \forall x \in \Omega_s, ~~ i = 1,\ldots,N_m.
\end{equation}
In some models employing the PBE, there is a third region $\Omega _l$  
(the Stern layer~\cite{Borukhov.I;Andelman.D;Orland.H1997}), a layer  
between $\Omega _m$ and $\Omega _s$. In the presence of a Stern layer,  
the parameter $\sigma$ in~(\ref{sigma})
%\break\vspace*{-12pt}\pagebreak \noindent
increases in value. Our
analysis and results can be easily generalized to this case as well. 

Some analytical solutions of the nonlinear PBE are known, but only for 
unrealistic structure geometries and usually only for linearizations
of the equation; cf.~\cite{Holst.M1994} for a collection of these 
solutions and for references to the large amount of literature on
analytical solutions to the PBE and similar equations. However, the
problem is highly nonlinear. Surface potentials of the linear and 
the nonlinear PBE differ by over an order of magnitude
\cite{Shestakov.A;Milovich.J;Noy.A2002}. Hence, using the nonlinear 
version of the PBE model is fundamentally important to accurately 
describe physical effects, and access to reliable and accurate 
numerical approximation techniques for the nonlinear PBE is critically 
important in this research area.

We finish this section by making some remarks about an alternative 
equivalent formulation of the PBE. It is well known
(cf.~\cite{Tanford.C1961,Holst.M1994}) that the PBE is formally 
equivalent to a coupling of two equations for the electrostatic 
potential in different regions $\Omega _m$ and $\Omega _s$ through the 
boundary interface.
This equivalence can be rigorously justified.
Inside $\Omega _m$, there are no ions. Thus 
the equation is simply the Poisson equation 
$$
 -\nabla \cdot (\varepsilon _m \nabla \tilde u) = \sum _{i=1}^{N_m}
q_i\delta _i \quad \hbox{ in } \Omega _m.
$$
In the solvent region $\Omega _s$, there are no atoms. Thus the 
density is given purely by the Boltzmann distribution 
$$
 -\nabla \cdot (\varepsilon _s \nabla \tilde u) + \bar \kappa ^2 \sinh (\tilde u)
= 0 \quad \hbox{ in } \Omega _s.
$$
These two equations are coupled together through the boundary 
conditions on the interface $\Gamma:=\partial \Omega _m=\partial 
\Omega _s \cap \Omega _m$: 
$$
[\tilde u] _{\Gamma}=0, \quad \text{ and }\quad
  \left [\varepsilon \frac{\partial \tilde u}{\partial
      n_{\Gamma}}\right ] _{\Gamma}=0,
$$
where $[f]|_{\Gamma}= \lim _{t\rightarrow 0}f(x+t 
n_{\Gamma})-f(x-tn_{\Gamma})$, with $n_{\Gamma}$ being the unit outward
normal direction of interface $\Gamma$. We will assume $\Gamma$ to be 
sufficiently smooth, say, of class $C^{2}$.

Solving the individual subdomain systems and coupling them through
the boundary, in the spirit of a nonoverlapping domain decomposition
method, is nontrivial due to the complicated boundary conditions and 
subdomain shapes. Approaches such as mortar-based finite element
methods to solve the coupled equations for linear or nonlinear PBE 
can be found in~\cite{Chen.W;Shen.Y;Xia.Q2005,Xie.D;Zhou.S2006}.

%sec 3
\section{Regularization of the continuous problem}\label{sec3}
In this section, we shall introduce a regularized version of the 
nonlinear PBE for both analysis and discretization purposes. We first 
transfer the original equation posed on the whole space to a truncated 
domain using an artificial boundary condition taken from an 
approximate analytical solution. Then we use the fundamental solution 
in the whole space to get rid of the singularities caused by $\delta$ 
distributions. We shall mainly focus on more difficult problems in 
three dimensions. Formulation and results in two dimensions are
similar and relatively easy.

Let $\Omega\subset \mathbb R^3$ with a convex and Lipschitz-continuous 
boundary $\partial \Omega$, and $\Omega_m \subset \Omega$. In the 
numerical simulation, for  simplicity, we usually choose $\Omega$
to be a ball or cube containing a molecule region. The solvent region is
chosen as $\Omega _s\cap \Omega$ and will be still denoted by $\Omega 
_s$. On $\partial \Omega$ we choose the boundary condition $\tilde 
u=g$, with
\begin{equation}\label{boundary}
g= \left( \frac{e_c^2}{k_B T} \right) \sum_{i=1}^{N_i}
\frac{e^{-\kappa|x-x_i|}}{\varepsilon_s|x-x_i|}.
\end{equation}
The boundary condition is usually taken to be induced by a known 
analytical solution to one of several possible simplifications of the
linearized PBE. Far from the molecule, such analytical solutions 
provide a highly accurate boundary condition approximation for the 
general nonlinear PBE on a truncation of $\mathbb R^3$. For example,
~(\ref{boundary}) arises from the use of the Green's function for the 
Helmholtz operator arising from linearizations of the 
Poisson--Boltzmann operator, where a single constant global dielectric 
value of $\varepsilon_s$ is used to generate the approximate boundary 
condition. (This is the case of a rod-like molecule approximation; 
cf.~\cite{Holst.M1994}.) Another approach to handling the boundary 
condition more accurately is to solve the PBE with boundary 
conditions such as (\ref{boundary}) on a large $\Omega$ (with a coarse
mesh) and then solve it in a smaller $\Omega$ (with a fine mesh) with 
the boundary condition provided by the earlier coarse mesh solution. The
theoretical justification of this approach can be found at 
\cite{Holst.M;Baker.N;Wang.F2000} using the two-grid theory 
\cite{Xu.J1996b}. We are not going to discuss more on the choice of 
the boundary condition in this paper.

Employing~(\ref{boundary}) we obtain the nonlinear PBE on a truncated 
domain:
\begin{align}\label{tPBE}
-\nabla \cdot (\varepsilon \nabla \tilde u)+\bar{\kappa} ^2\sinh
(\tilde u)&=\sum _{i=1}^{N_m}q_i\delta _i & \hbox{in
}\;\Omega,\\
\tilde u&=g & \hbox{on }\;\partial \Omega.\label{eq3.3}
\end{align}
This is, in most respects, a standard boundary-value problem for a 
nonlinear second-order elliptic partial differential equation. However,
the right side contains a linear combination of $\delta$ distributions,
which individually and together are not in $H^{-1}(\Omega)$; thus we 
cannot apply standard techniques such as classical potential 
theory. This has at times been the source of some confusion in the 
molecular modeling community, especially with respect to the design of 
convergent numerical methods. More precisely, we will see shortly 
that the solution to the nonlinear Poisson--Boltzmann equation is 
simply not globally smooth enough to expect standard numerical 
methods (currently used by most PBE simulators) to produce 
approximations that converge to the solution to the PBE in the limit 
of mesh refinement.

In order to gain a better understanding of the properties of solutions 
to the nonlinear PBE, primarily so that we can design new provably 
convergent numerical methods, we shall propose a decomposition of the 
solution to separate out the singularity caused by the $\delta$ 
distributions. This decomposition will turn out to be the key idea that 
will allow us to design discretization techniques for the nonlinear 
PBE which have provably good approximation properties and, based on
this, also design a new type of adaptive algorithm which is provably 
convergent for the nonlinear PBE.

We now give this decomposition. It is well known that the function
$$
G_i =\frac{q_i}{\varepsilon _m}\frac{1}{|x-x_i|}
$$
solves the equation
$$
 -\nabla \cdot (\varepsilon _m\nabla G_i) = q_i\delta _i \quad \hbox{
 in }\; \mathbb R^3.
$$
We thus decompose the unknown $\tilde u$ as an unknown smooth function
$u$ and a known singular function $G$:
$$
\tilde u = u + G,
$$
with
\begin{equation}
\label{G}
G=\sum _{i=1}^{N_m}G_i.
\end{equation}

Substituting the decomposition into (\ref{tPBE}), we then obtain
\begin{align}
\label{RPBE1}
-\nabla \cdot (\varepsilon \nabla u)+\bar{\kappa} ^2\sinh (u+G) &=
\nabla \cdot ((\varepsilon -\varepsilon _m)\nabla G) & \hbox{ in
}\; \Omega, \\
\label{RPBE2}
u&=g- G & \hbox{ on }\; \partial \Omega,
\end{align}
and call it the RPBE. The
singularities of the $\delta$ distribtuions are transferred to $G$, which 
then exhibits degenerate behavior at each $\{x_i\}\subset \Omega _m$. 
At those points, both $\sinh G(x_i)$ and $\nabla G(x_i)$ exhibit 
blowup. However, since $G$ is known analytically, one avoids having to 
build numerical approximations to $G$. Moreover, both of the 
coefficients $\bar{\kappa}$ and $\varepsilon-\varepsilon _m$ are zero 
inside $\Omega _m$, where the blowup behavior arises. Due to this
cutoff nature of coefficients, the RPBE is a mathematically
well defined nonlinear second-order elliptic equation for the
regularized solution $u$ with the source term in $H^{-1}$. We give a 
fairly standard argument in the next section to show that it also 
admits a unique solution $u\in H^1$, even though the original solution 
$\tilde u\notin H^1$ due to the singularities present in $G$. In the 
remainder of the paper we shift our focus to establishing additional 
estimates and developing an approximation theory to guide the design 
of convergent methods, both nonadaptive and adaptive.

Before moving on, it is useful to note that, away from $\{x_i\}$, the
function $G$ is smooth. In particular, we shall make use of the fact that
$G\in C^{\infty}(\Omega _s)\cap C^{\infty}(\Gamma)\cap 
C^{\infty}(\partial \Omega)$ in the later analysis. Also, a key 
technical tool will be a further decomposition of the regularized 
solution $u$ into linear and nonlinear parts, $u=u^l+u^n$, where $u^l$ 
satisfies
\begin{align}
\label{RPBE1_linear}
-\nabla \cdot (\varepsilon \nabla u^l) &= \nabla \cdot ((\varepsilon
-\varepsilon _m)\nabla G) & \hbox{ in } \Omega, \\
\label{RPBE2_linear}
u^l&=0 & \hbox{ on } \partial \Omega,
\end{align}
and where $u^n$ satisfies
\begin{align}
\label{RPBE1_nonlinear}
-\nabla \cdot (\varepsilon \nabla u^n)+\bar{\kappa} ^2\sinh
(u^n+u^l+G) &= 0 & \hbox{ in } \Omega, \\
\label{RPBE2_nonlinear}
u^n&=g- G & \hbox{ on } \partial \Omega.
\end{align}

%sec 4
\section{Existence and uniqueness}\label{sec4}
In this section we shall discuss the existence and uniqueness of the 
solution of the continuous RPBE. The arguments we use in this section 
appear essentially in~\cite{Holst.M1994}, except there the PBE was 
artificially regularized by replacing the delta distributions with 
$H^{-1}$-approximations directly rather than being regularized
through a singular function expansion.

We first write out the weak formulation. Since $\Delta G=0$ away 
from $\{x_i\}$, through integration by parts we get the weak 
formulation of RPBE: Find
$$
u\in M:=\{v\in H^1(\Omega) \, | \, e^v, e^{-v}\in L^2(\Omega _s),\;
\hbox{ and } \; v=g-G \;\hbox{ on } \; \partial \Omega\}
$$
such that
\begin{equation}\label{weak}
A(u, v)+(B(u),v)+\langle f_G,v\rangle =0 \quad \forall v\in
H_0^1(\Omega),
\end{equation}
where
\begin{itemize}
\item $A(u,v)=(\varepsilon \nabla u,\nabla u)$,
\item $(B(u),v)=({\bar \kappa} ^2\sinh(u+G),v)$, and
\item $\langle f_G,v\rangle= \int_{\Omega}(\varepsilon -\varepsilon_m)
\nabla G\cdot \nabla v.$
\end{itemize}

Let us define the energy on $M$:
$$
E(w)= \int _{\Omega}\frac{\varepsilon}{2}|\nabla w|^2+\bar{\kappa}
^2\cosh(w+G)+\langle f_G, w\rangle.
$$
It is easy to characterize the solution of (\ref{weak}) as the 
minimizer of the energy.

%lemma 4.1
\begin{lemma}\label{lemma4.1}
If $u$ is the solution of the optimization problem, i.e.,
$$
E(u)=\inf _{w\in M}E(w),
$$
then $u$ is the solution of {\rm (\ref{weak})}.
\end{lemma}

\begin{proof}
For any $v\in H_0^1(\Omega)$ and any $t\in \mathbb R$, the function 
$F(t)=E(u+tv)$ attains the minimal point at $t=0$, and thus $F'(0)=0$,
which gives the desired result.%
\qquad\end{proof}

We now recall some standard variational analysis on the existences of 
the minimizer. In what follows we suppose $S$ is a set in some Banach
space $V$ with norm $\|\cdot\|$, and $J(u)$ is a functional defined on 
$S$. $S$ is called {\it weakly sequential compact} if, for any sequence
$\{u_k\}\subset S$, there exists a subsequence $\{u_{k_i}\}$ such that 
$u_{k_i}\rightharpoonup u\in S$, where $\rightharpoonup$ stands for 
the convergence in the weak topology. For any $u_k\rightharpoonup u$, 
if $J(u_k)\rightarrow J(u)$, we say $J$ is {\it weakly continuous} at 
$u$; if
$$
J(u)\leq \liminf _{k\rightarrow \infty}J(u_k),
$$
we say $J$ is {\it weakly lower semicontinuous} (w.l.s.c.) at
$u$. The following theorem can be proved by the definition easily. 

%theorem 4.2
\begin{theorem}\label{theorem4.2}
If
\begin{enumerate}
\item[{\rm 1.}] $S$ is weakly sequential compact, and 
\item[{\rm 2.}] $J$ is weakly lower semicontinuous on $S$,
\end{enumerate}
then there exists $u\in S$ such that
$$
J(u)=\inf _{w\in S}J(w).
$$
\end{theorem}
\unskip

We shall give conditions for the weakly sequential compactness and 
weakly lower semicontinuity. First we use the fact that a bounded set
in a reflexive Banach space is weakly sequential compact. 

%lemma 4.3
\begin{lemma}\label{lemma4.3}
One has the following results:
\begin{enumerate}
\item[{\rm 1.}] The closed unit ball in a reflexive Banach space $V$ 
is weakly sequential compact.
\item[{\rm 2.}] If $\lim _{\|v\|\rightarrow \infty}J(v)=\infty$, then 
$$
\inf_{w\in V} J(w)=\inf _{w\in S} J(w).
$$
\end{enumerate}
\end{lemma}
%\unskip

The next lemma concerns when the functional is w.l.s.c. The proof
can be found at~\cite{Yosida.K1980}.

%lemma 4.4
\begin{lemma}\label{lemma4.4}
If $J$ is a convex functional on a convex set $S$ and $J$ is G\^ateaux
differentiable, then $J$ is w.l.s.c.\ on $S$.
\end{lemma}

Now we are in the position to establish the existence and uniqueness 
of solutions to the RPBE.

%theorem 4.5
\begin{theorem}\label{theorem4.5}
There exists a unique $u\in M\subset H^1(\Omega)$ such that
$$
E(u)=\inf _{w\in M}E(w).
$$
\end{theorem}
\unskip

\begin{proof}
It is easy to see $E(w)$ is differentiable in $M$ with
$$
\langle DE(u),v\rangle = A(u, v)+(B(u),v)+\langle f_G,v\rangle.
$$
To prove the existence of the minimizer, we need only to verify that
\begin{enumerate}
\item $M$ is a convex set,
\item $E$ is convex on $M$, and
\item $\lim _{\|v\|_1\rightarrow \infty}E(v)=\infty$.
\end{enumerate}
The verification of (1) is easy and thus skipped here. (2) comes from 
the convexity of functions $x^2$ and $\cosh (x)$. Indeed $E$ is {\em 
strictly} convex. (3) is a consequence of the inequality
\begin{equation}\label{corcevity}
E(v)\geq C(\varepsilon,\bar \kappa) \|v\|^2_1+ C(G,g),
\end{equation}
which can be proved as following.  First, by Young's inequality we have for
any $\delta > 0$
$$
\langle f_G, v\rangle \leq \varepsilon _s\|\nabla G\|_{\Omega_s}
\|\nabla v\|_{\Omega _s}\leq \frac{1}{\delta} \|\nabla G\|^2_{\Omega _s} +
 \delta \varepsilon^2_s\| \nabla_v\|^2_{\Omega _s}.
$$
Since $\cosh(x) \geq 0$, we have then $E(v)\geq C(\varepsilon,\bar \kappa)\|\nabla
v\|^2- (1/\delta)\|\nabla G\|^2_{\Omega _s}$, where we can
ensure $C(\varepsilon, \bar \kappa) > 0$ if $\delta$ is chosen to be
sufficiently small.  Then using norm equivalence on
$M$, we get (\ref{corcevity}). The uniqueness of the minimizer comes
from the strict convexity of $E$.%
\qquad\end{proof}

%sec 5
%\section{Continuous a priori \boldmath $L^{\infty}$-estimates\unboldmath}\label{sec5}
\section{Continuous a priori $L^{\infty}$-estimates}\label{sec5}
In this section, we shall derive a priori $L^{\infty}$-estimates
of the solution of the RPBE. The main result of this section is the 
following theorem.

%theorem 5.1
\begin{theorem}\label{th:max}
Let $u$ be the weak solution of RPBE in $H^1(\Omega)$. Then $u$ is 
also in $L^{\infty}(\Omega)$.
\end{theorem}

Note that we cannot apply the analysis of 
\cite{Jerome.J1985,Kerkhoven.T;Jerome.J1990} directly to the RPBE,
since the right side $f_G\in H^{-1}(\Omega)$ and does not lie in
$L^{\infty}(\Omega)$ as required for use of these results. We shall 
overcome this difficulty through further decomposition of $u$ into 
linear and nonlinear parts.

Let $u=u^l+u^n$, where $u^l\in H_0^1(\Omega)$ satisfies the linear 
elliptic equation (the weak form of 
(\ref{RPBE1_linear})--(\ref{RPBE2_linear}))
\begin{equation}\label{linear}
A(u^l,v)+\langle f_G, v\rangle=0 \quad \forall v\in H_0^1(\Omega)
\end{equation}
and where $u^n\in M$ satisfies the nonlinear elliptic equation (the
weak form of (\ref{RPBE1_nonlinear})--(\ref{RPBE2_nonlinear}))
\begin{equation}\label{nonlinear}
A(u^n,v)+ (B(u^n+u^l), v) = 0 \quad \forall v\in H_0^1(\Omega).
\end{equation}

Theorem \ref{th:max} then follows from the estimates of $u^l$ and 
$u^n$ in Lemmas \ref{lm:ul} and \ref{lm:nonlinear}; cf.\ (\ref{eq:ul})
and (\ref{eq:un}).
%\pagebreak

%lemma 5.2
\begin{lemma}\label{lm:ul}
Let $u^l$ be the weak solution of {\rm (\ref{linear})}. Then
\begin{equation}
\label{eq:ul}
u^l\in L^{\infty}(\Omega).
\end{equation}
\end{lemma}
\unskip

\begin{proof}
Since $\Delta G=0$ in $\Omega _s$, using integral by parts we 
can rewrite the functional $f_G$ as
$$
\langle f_G, v\rangle = ((\varepsilon -\varepsilon _m)\nabla G,
\nabla v) = \left([\varepsilon]\frac{\partial G}{\partial n_{\Gamma}},v\right)_{\Gamma},
$$
where $[\varepsilon]=\varepsilon _s-\varepsilon _m$ is the jump of 
$\varepsilon$ at the interface. We shall still use $f_G$ to denote the 
smooth function $[\varepsilon]\frac{\partial G}{\partial n_{\Gamma}}$ 
on $\Gamma$.

It is easy to see that the linear equation (\ref{linear}) is the weak 
formulation of the elliptic interface problem 
$$
 -\nabla \cdot (\varepsilon \nabla u^l)=0 \text{ in } \Omega \quad
[u^l]=0, \; \left [\varepsilon \frac{\partial u^l}{\partial n}\right
]=f_G \text{ on }\Gamma, \quad \text{and}\; u=0  \mbox{ on}\,
\partial \Omega.
$$
Since $f_G\in C^{\infty}(\Gamma)$ and $\Gamma\in C^{2}$, by the 
regularity result of the elliptic interface problem 
\cite{Babuska.I1970,Bramble.J;King.J1996,Chen.Z;Zou.J1998,Savare.G1998}, 
we have $u^l\in H^2(\Omega _m)\cap H^2(\Omega _s)\cap 
H_0^1(\Omega)$. In particular by the embedding theorem we conclude that
$u^l\in L^{\infty}(\Omega).$%
\qquad\end{proof}

To derive a similar estimate for the nonlinear part $u^n$, we define
\begin{align*}
\alpha ' & = \arg \max_c \Big( \bar{\kappa} ^2\sinh( c + \sup_{x\in
 \Omega_s} (u^l+ G) ) \leq 0\Big ), & \alpha & = \min \Big(\alpha ',
\inf _{\partial \Omega}(g-G)\Big), \\
\beta ' & = \arg \min_c \Big(
\bar{\kappa} ^2\sinh( c + \inf_{x\in \Omega_s} (u^l+G) ) \geq 0\Big ),
& \beta & = \max \Big(\beta ', \sup _{\partial \Omega}(g-G)\Big).
\end{align*}
The next lemma gives the a priori $L^{\infty}$-estimate of
$u^n$.

%lemma 5.3
\begin{lemma}\label{lm:nonlinear}
Let $u^n$ be the weak solution of {\rm (\ref{nonlinear})}. Then 
$\alpha \leq u^n\leq \beta$, and thus 
\begin{equation}
\label{eq:un}
u^n\in L^{\infty}(\Omega).
\end{equation}
\end{lemma}
\unskip

\begin{proof}
We use a cutoff-function argument similar to that used in 
\cite{Jerome.J1985}. Since the boundary condition $g-G\in 
C^{\infty}(\partial \Omega)$, we can find a $u_D\in H^1(\Omega)$ 
such that $u_D=g-G$ on $\partial \Omega$ in the trace sense, or more 
precisely
$$
T u_D = g-G,
$$
where $T:\Omega \mapsto \partial\Omega$ is the trace operator.  Then 
the solution can be written $u^n=u_D+u_0$, with $u_0\in H_0^1(\Omega)$.
Let $\overline{\phi} = (u^n-\beta)^+=\max (u^n-\beta, 0)$
and $\underline{\phi} = (u^n-\alpha)^-=\min (u^n-\alpha, 0)$. 
Then from
\begin{eqnarray*}
0 & \le & \overline{\phi} =(u^n-\beta)^+=(u_D+u_0-\beta)^+\leq (u_D-\beta)^+ + u_0^+,\\
0 & \ge & \underline{\phi} =(u^n-\alpha)^-=(u_D+u_0-\alpha)^-\geq (u_D-\alpha)^- + u_0^-,
\end{eqnarray*}
and
\begin{eqnarray*}
0 & \le & T\overline{\phi} \le T(u_D-\beta)^+ + Tu_0^+ = 0,
\\
0 & \ge & T\underline{\phi} \ge T(u_D-\alpha)^- + Tu_0^- = 0,
\end{eqnarray*}
we conclude that both $\overline{\phi}, \underline{\phi} \in H_0^1(\Omega)$.  
Thus for either $\phi=\overline{\phi}$ or $\phi=\underline{\phi}$, we have 
$$
(\varepsilon \nabla u^n,\nabla \phi)+(\bar{\kappa}
^2\sinh(u^n+u^l+G),\phi) = 0.
$$
Note that $\overline{\phi} \ge 0$ in $\Omega$ and its support is the set 
$\overline{\mathcal{Y}} = \{x \in \bar \Omega\, | \, u^n(x)\geq\beta\}$.  
On $\overline{\mathcal{Y}}$, we have 
$$
\bar{\kappa} ^2\sinh(u^n+u^l+G) \ge \bar{\kappa} ^2\sinh \Big(\beta ' +
\inf_{x\in \Omega_s} (u^l + G) \Big) \ge 0.
$$
Similarly, $\underline{\phi} \le 0$ in $\Omega$ with support set 
$\underline{\mathcal{Y}} = \{x \in \bar \Omega\, | \, u^n(x)\leq\alpha\}$.  
On $\underline{\mathcal{Y}}$, we now have 
$$
\bar{\kappa} ^2\sinh(u^n+u^l+G) \le \bar{\kappa} ^2\sinh \Big(\alpha ' +
\inf_{x\in \Omega_s} (u^l + G) \Big) \le 0.
$$
Together this implies
$$
0 \ge (\varepsilon \nabla u^n,\nabla \phi)=(\varepsilon \nabla
(u^n-\beta),\nabla \phi)=\varepsilon \|\nabla \phi\|^2 \ge 0
$$
for either $\phi=\overline{\phi}$ or $\phi=\underline{\phi}$.
Using the Poincare inequality we have finally
$$
0 \le \|\phi\| \lesssim \|\nabla \phi\| \le 0,
$$
giving $\phi =0$, again for either $\phi=\overline{\phi}$ or 
$\phi=\underline{\phi}$. Thus $\alpha \leq u^n\leq \beta$ in $\Omega$.%
\qquad
\end{proof}

%sec 6
\section{Finite element methods for the regularized Poisson--Boltzmann 
equation}\label{sec6}
In this section we shall discuss the finite element discretization of 
RPBE using linear finite element spaces $V^h_D$ and prove the
existence and uniqueness of the finite element approximation 
$u_h$. Furthermore, under some assumptions on the grids we shall 
derive a priori $L^{\infty}$-estimates for $u_h$ and use these
to prove that $u_h$ is a quasi-optimal approximation of $u$ in 
the $H^1$ norm in the sense that
\begin{equation}\label{mainestimate}
\|u-u_h\|_1\lesssim \inf _{v_h\in V^h_D}\|u-v_h\|_1.
\end{equation}

While the term on the left in~(\ref{mainestimate}) is in general 
difficult to analyze, the term on the right represents the fundamental 
question addressed by classical approximation theory in normed spaces, 
of which much is known. To bound the term on the right from above, one 
picks a function in $V^h_D$ which is particularly easy to work with, 
namely, a nodal or generalized interpolant of $u$, and then one employs
standard techniques in interpolation theory. Therefore, it is clear 
that the importance of approximation results such 
as~(\ref{mainestimate}) are that they completely separate the details 
of the Poisson--Boltzmann equation from the approximation theory, 
making available all known results on finite element interpolation of 
functions in Sobolev spaces (cf.~\cite{Ciarlet.P1978}). 

Now we assume $\Omega$ can be triangulated exactly (e.g., $\Omega$ is a
cube) with a shape regular and conforming (in the sense of 
\cite{Ciarlet.P1978}) triangulation $\mathcal T_h$. Here $h=h_{\max}$ represents 
the mesh size which is the maximum diameter of elements in the 
triangulation. We further assume in the triangulation that the discrete
interface $\Gamma _h$ approximates the known interface $\Gamma$ to the 
second order, i.e., $d(\Gamma, \Gamma _h)\leq Ch^2$.

Given such a triangulation $\mathcal T_h$ of $\Omega$, we construct 
the linear finite element space $V^h:=\{v\in H^1(\Omega), v|_{\tau}\in 
\mathcal P_1(\tau)\ \forall \tau \in \mathcal T_h\}$. Since the
boundary condition $g-G\in C^{\infty}(\partial \Omega)$, we can find a 
$u_D\in H^1(\Omega)$ such that $u_D=g-G$ on $\partial \Omega$ in the 
trace sense. Then the solution can be uniquely written as 
$u=u_D+u_{0}$, with $u_0\in H_0^1$. Thus we will use $\displaystyle
H_D^1(\Omega):=H_0^1(\Omega)+u_D$ to denote the affine space with 
a specified boundary condition and $\displaystyle V^h_D=V^h\cap
H_D^1(\Omega)$ to denote the finite element affine space of 
$H_D^1(\Omega)$. Similarly $V^h_0=V^h\cap H_0^1(\Omega)$. Here to
simplify the analysis the boundary condition is assumed to be
represented exactly.

Recall that the weak form of RPBE is
\begin{equation} \label{eqn:gal_cont}
\hspace*{6pt}\text{Find}~u \in H_D^1(\Omega) ~~\mbox{such that (s.t.) } A(u,v) + (
B(u),v) + \langle f_G, v \rangle = 0  ~~\forall v \in
H_0^1(\Omega).\hspace*{-8pt}
\end{equation}

We are interested in the quality of the finite element approximation:
\begin{equation}
\label{eqn:gal_disc}
\text{Find}~u_h \in V^h_D~~\text{ s.t.}~ A(u_h,v_h) + ( B(u_h),v_h
) + \langle f_G, v \rangle = 0 ~~\forall v_h \in V^h_0.
\end{equation}
It is easy to show that the finite element approximation $u_h$ is the
minimizer of $E$ in $V^h_D$, i.e., $E(u_h)=\inf _{v_h\in V^h_D}E(v_h)$.
Then the existence and uniqueness follows from section \ref{sec3} 
since $V^h_D$ is convex. As in the continuous setting, it will be 
convenient to split the discrete solution to the RPBE into linear 
and nonlinear parts $u_h = u_h^l + u_h^n$, where $u_h^l$ and 
$u_h^n$ satisfy, respectively,
\begin{equation}
\label{eqn:gal_disc_linear}
\text{Find}~u_h^l \in V^h_0~\text{s.t.}~ A(u_h^l,v_h) + \langle f_G, v
\rangle = 0 ~~\forall v_h \in V^h_0,
\end{equation}
\begin{equation}
\label{eqn:gal_disc_nonlinear}
\text{Find}~u_h^n \in V^h_D~\text{s.t.}~ A(u_h^n,v_h) + \langle
B(u_h^n+u_h^l),v_h \rangle =0 ~~\forall v_h \in V^h_0.
\end{equation}

%6.1
\subsection{Quasi-optimal a priori error estimate}\label{sec6.1}
We begin with the following properties of the bilinear form $A$ and 
and operator $B$.

%lemma 6.1
\begin{lemma}\label{lemma6.1}
{\rm 1.} The bilinear form $A(u,v)$ satisfies the coercivity and 
continuity conditions. That is, for $u, v\in H^1(\Omega)$
$$
\|u\|_1^2\lesssim A(u,u),\; \hbox{ and }\; A(u,v)\lesssim
\|u\|_1\|v\|_1.
$$

{\rm 2.} The operator $B$ is monotone in the sense that
$$
(B(u)-B(v),u-v)\geq \bar {\kappa}^2\|u-v\|^2\geq 0.
$$

{\rm 3.} The operator $B$ is bounded in the sense that for $u, v\in 
 L^{\infty}(\Omega), w\in L^2(\Omega)$,
$$
(B(u)-B(v),w)\leq C\|u-v\|\|w\|.
$$
\end{lemma}
\unskip

\begin{proof}
The proof of (1) and (2) is straightforward. We now prove (3). By the 
mean value theorem, there exists $\theta \in (0,1)$ such that 
$$
B(u)-B(v)={\bar \kappa}^2 \cosh (\theta u+(1-\theta)v+G)(u-v).
$$
Then by the convexity of $\cosh$ and the fact that $u, v\in L^{\infty}(\Omega),
G\in C^{\infty}(\Omega_s)$, we get
\begin{align*}
\|\cosh (\theta u+(1-\theta)v+G)\|_{\infty,\Omega _s} \leq
 \|\cosh (u+G)\|_{\infty,\Omega _s}+\|\cosh (v+G)\|_{\infty,\Omega
  _s}\leq C.
\end{align*}
The desired result then follows since $B(\cdot)$ is nonzero only in
$\Omega_s$.%
\qquad\end{proof}

%theorem 6.2
\begin{theorem}\label{theorem6.2}
Let $u$ and $u_h$ be the solution of RPBE and its finite element 
approximation, respectively. When $u_h$ is uniformly bounded, we have 
$$
\|u-u_h\|_1\lesssim \inf _{v_h\in V^h}\|u-v_h\|_1.
$$
\end{theorem}
\unskip

\begin{proof}
By the definition, the error $u-u_h$ satisfies
$$
A(u-u_h, w_h) + (B(u)-B(u_h), w_h)=0 \quad \forall w_h\in V^h_0.
$$
We then have, for any $v_h\in V_D^h$,
\begin{eqnarray*}
\|u-u_h\|_1^2&\lesssim &A(u-u_h,u-u_h)=A(u-u_h,u-v_h)+A(u-u_h,
v_h-u_h)\\ &\lesssim &\|u-u_h\|_1\|u-v_h\|_1-(B(u)-B(u_h),v_h-u_h).
\end{eqnarray*}
The second term on the right side is estimated by
\begin{eqnarray*}
-(B(u)-B(u_h),v_h-u_h)&=&-(B(u)-B(u_h),u-u_h)+(B(u)-B(u_h),u-v_h)\\
&\leq &(B(u)-B(u_h),u-v_h)\\ &\lesssim & \|u-u_h\|_1\|u-v_h\|_1.
\end{eqnarray*}
Here we make use of the monotonicity of $B$ in the second step and the
boundness of $B$ in the third step. In summary we obtain for any 
$v_h\in V_D^h$
$$
\|u-u_h\|_1\lesssim \|u-v_h\|_1,
$$
which leads to the desired result by taking the infimum.%
\qquad\end{proof}

%6.2
%\subsection{Discrete a priori \boldmath $L^{\infty}$-estimates\unboldmath}\label{sec6.2}
\subsection{Discrete a priori $L^{\infty}$-estimates}\label{sec6.2}
We now derive $L^{\infty}$-estimates of the finite element 
approximation $u_h$. To this end, we have to put assumptions on the 
grid. Let $(a_{ij})$ denote the matrix of the elliptic operator 
$(\varepsilon \nabla u, \nabla v)$, i.e., $a_{i,j}=A(\varphi _i, \varphi_j)$.
Two nodes $i$ and $j$ are adjacent if there is an edge connecting them.

(A1) The off-diagonal term $a_{i,j},$ $i, j$ are
adjacent, satisfies
$$
a_{i,j}\leq -\frac{\rho}{h^2}\sum _{e_{i,j}\subset T}|T|, \quad
\text{with }\; \rho >0.
$$

We now give example grids satisfying (A1). In three dimensions, to 
simplify the generation of the grid, we choose $\Omega$ as a cube and 
divide into small cubes with length $h$. For each small cube, we 
divide it into 5 tetrahedra; see Figure~\ref{5tet} for a prototype of the
triangulation of one cube. Neighbor cubes are triangulated in the same 
fashion (with different reflection to make the triangulation 
conforming). By the formula of the local stiffness matrix in 
\cite{Kerkhoven.T;Jerome.J1990,Xu.J;Zikatanov.L1999}, it is easy to 
verify that the grids will satisfy assumption (A1). We comment that 
the uniform grid obtained by dividing each cube into 6 tetrahedra
will not satisfy the assumption (A1), since in this case if $i,j$ are
vertices of diagonal of some cube, then $a_{ij}=0$.

%figure 6.1
\begin{figure}[tbp!]
\begin{center}
\begin{pdfpic}
\psset{unit=0.8cm}
\begin{pspicture}*(-0.1,-0.1)(4.1,4.1)
%%%\psgrid[subgriddiv=5,griddots=10,gridlabels=0.25cm](0,0)(-1,-1)(10,10)
\pspolygon(0,0)(3,0)(3,3)(0,3)
%%%\rput(1,1){\pspolygon(0,0)(3,0)(3,3)(0,3)}
\psline(0,3)(1,4)(4,4)(3,3)(1,4)
\psline(3,0)(4,1)(4,4)
\psline(0,0)(3,3)(4,1)
\psline[linestyle=dashed](0,0)(1,1)(1,4)
\psline[linestyle=dashed](1,1)(4,1)
\pspolygon[linestyle=dashed](0,0)(4,1)(1,4)
\end{pspicture}
\end{pdfpic}
\caption{Divide a cube into $5$ tetrahedra.\label{5tet}}
\end{center}
\end{figure}
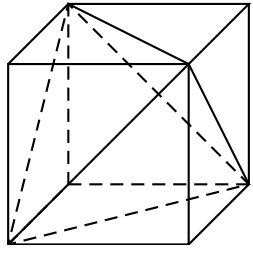

%theorem 6.3
\begin{theorem}\label{theorem6.3}
In general dimension $\mathbb R^d, d\geq 2$, with assumption {\rm (A1)} 
and $h$ sufficiently small, the finite element approximation $u_h$ of RPBE 
satisfies
$$
\|u_h\|_{\infty}\leq C,
$$
where $C$ is independent of $h$.
\end{theorem}

%\newpage

\begin{proof}
We shall use the decomposition $u_h=u_h^n+u_h^l$. By the regularity 
result~\cite{Savare.G1998}, we know $u^l\in B^{3/2}_{2,\infty}(\Omega)$
and thus obtain a priori estimate on quasi-uniform grids
$$
\|u^l-u^l_h\|_{\infty}\leq Ch^{s}_{\max} \leq C {\rm diam}(\Omega)^s\; 
\text{ for some }s\in (0,3/2).
$$
This implies that $\|u^l_h\|_{\infty}\leq \|u^l\|_{\infty} + 
\|u^l-u^l_h\|_{\infty} \leq C$ is uniformly bounded with respect to 
$h_{\max}$. The estimate of $u^n_h$ follows from Theorem 3.3 
in~\cite{Kerkhoven.T;Jerome.J1990}, where the grid assumption (A1) is
used.%
\qquad\end{proof}

In two dimensions, we can relax the assumption on the grid and obtain 
a similar result. Later we will see that, due to this relaxation, the local
refinement in two dimensions is pretty simple.

(A1$'$) The off-diagonal terms $a_{i,j}\leq 0, j\neq i$;
i.e., the stiffness matrix corresponding to $A(\cdot, \cdot)$ is an
M-matrix.

%theorem 6.4
\begin{theorem}\label{theorem6.4}
For a two-dimensional triangulation satisfying {\rm (A1$'$)}, the finite
element approximation $u_h$ of RPBE is bounded, i.e.,
$$
\|u_h\|_{\infty}\leq C.
$$
\end{theorem}
\unskip

\begin{proof}
Similarly $\|u^l_h\|_{\infty}\leq C$ is uniformly bounded. In two 
dimensions the estimate of $u^n_h$ follows from Theorem 3.1 
in~\cite{Kerkhoven.T;Jerome.J1990}, where the grid assumption (A1) 
is\linebreak used.%
\qquad\end{proof}

%sec 7
\section{Convergence of adaptive finite element approximation}\label{sec7}
In this section, we shall follow the framework presented in 
\cite{Verfurth.R1994,Verfurth.R1996} to derive an  a posteriori
error estimate. Furthermore we shall present an adaptive method 
through local refinement based on this error estimator and prove that it
will converge. The a priori $L^{\infty}$-estimates of the
continuous and discrete problems derived in the previous sections play an
important role here.

%7.1
\subsection{A posteriori error estimate}\label{sec7.1}
There are several approaches to adaptive error control, among which 
the one based on  a posteriori error estimation is usually the
most effective and most general. Although most existing work on  a
posteriori estimates has been for linear problems, extensions to
the nonlinear case can be made through linearization. For example, 
consider the nonlinear problem
\begin{equation}\label{F}
F(u)=0, \quad F\in C^1(\mathcal B_1, B_2^*),
\end{equation}
where the Banach spaces $\mathcal B_1$ and $\mathcal B_2$ are,
e.g., Sobolev spaces and where $\mathcal B^*$ denotes the dual space
of $\mathcal B$. Consider now also a discretization of
(\ref{F})
\begin{equation}\label{Fh}
F_h(u_h)=0, \quad F_h\in C^0(U_h, V_h^*),
\end{equation}
where $U_h\subset \mathcal B_1$ and $V_h\subset \mathcal B_2$. For the 
RPBE and a finite element discretization, the function spaces would be 
taken to be $\mathcal B_1=\mathcal B_2=H_0^1(\Omega)$. The nonlinear 
residual $F(u_h)$ can be used to estimate the error through the use of 
a linearization inequality
\begin{equation}
\label{eqn:verf}
C_1 \|F(u_h)\|_{\mathcal B_2^*} \le \| u - u_h \|_{\mathcal B_1} \le
C_2 \|F(u_h)\|_{\mathcal B_2^*}.
\end{equation}
See, for example,~\cite{Verfurth.R1994} for a proof of this linearization
result under weak assumptions on $F$. The estimator is then based on an 
upper bound on the dual norm of the nonlinear residual on the right 
in~(\ref{eqn:verf}).

In this section, to show the main idea, we will assume 
$F_h(u_h)=F(u_h)$ by making the following assumption on the 
grid.

(A2) The smooth interface $\Gamma$ is replaced by its
discrete approximation $\Gamma _h$ such that $\varepsilon$ and 
$\bar{\kappa}$ are piecewise constants on each element of the 
triangulation $\mathcal T_h$.

In our setting of the weak formulation, we need to estimate 
$\|F(u_h)\|_{-1,\Omega}$. To this end, we first introduce quite a bit 
of notation. We assume that the $d$-dimensional domain $\Omega$ has 
been exactly triangulated with a set $\mathcal T_h$ of shape-regular 
$d$-simplices (the finite dimension $d$ is arbitrary, not restricted 
to $d\leq 3$, throughout this discussion). A family of simplices will 
be referred to here as shape-regular in the sense of 
\cite{Ciarlet.P1978}.

It will be convenient to introduce the following notation:
\begin{center}
\begin{tabular}{lcl}
$\mathcal T_h$ & = & the set of shape-regular simplices triangulating the
 domain $\Omega$. \\

$\mathcal {N}(\tau)$ & = & the union of faces contained in simplex set
 $\tau$ lying on $\partial \Omega$. \\

$\mathcal {I}(\tau)$ & = & the union of faces contained in simplex set
 $\tau$ not in $\mathcal {N}(\tau)$. \\

$\mathcal {F}(\tau)$ & = & $\mathcal {N}(\tau) \cup \mathcal
 {I}(\tau)$. \\

$\mathcal F$ & = & $\cup _{\tau \in \mathcal T_h} \mathcal F(\tau)$.\\

$\omega_{\tau}$ & = & $~\bigcup~ \{~ \tilde{\tau} \in \mathcal T_h ~|~
 \tau \bigcap \tilde{\tau} \ne \varnothing, ~\text{where}~\tau \in
 \mathcal T_h ~\}$. \\

$\omega_S$ & = & $~\bigcup~ \{~ \tilde{\tau} \in \mathcal T_h ~|~ S
 \bigcap \tilde{\tau} \ne \varnothing, ~\text{where}~S \in \mathcal {F}
 ~\}$. \\

$h_{\tau}$ & = & the diameter of the simplex $\tau$. \\

$h_S$ & = & the diameter of the face $S$. \\
\end{tabular}
\end{center}
When the argument to one of the face set functions $\mathcal {N}$, 
$\mathcal {I}$, or $\mathcal {F}$ is in fact the entire set of 
simplices, we will leave off the explicit dependence on $\mathcal {S}$ 
without danger of confusion. Finally, we will also need some notation 
to represent discontinuous jumps in function values across faces 
interior to the triangulation. For any face $S \in \mathcal {N}$, let 
$n_S$ denote the unit outward normal; for any face $S\in \mathcal 
{I}$, take $n_S$ to be an arbitrary (but fixed) choice of one of the 
two possible face normal orientations. Now, for any $v \in 
L^2(\Omega)$ such that $v \in C^0(\tau)\ \forall \tau \in \mathcal
T_h$, define the {\em jump function}:
$$
[v]_S(x) = \lim_{t \rightarrow 0^+} v(x+t n_S) - \lim_{t
 \rightarrow 0^-} v(x-t n_S).
$$

We now define our  a posteriori error estimator
\begin{equation}\label{eqn:estimator}
\eta_{\tau}^2(u_h) = h_{\tau}^2 \| B(u_h)\|_{0,\tau}^2 + \frac{1}{2}
\sum_{S \in \mathcal {I}(\tau)} h_S\| \left[ n_S \cdot (\varepsilon
 \nabla u_h + (\varepsilon- \varepsilon _m)\nabla G )\right]_S
\|_{0,S}^2,
\end{equation}
and the oscillation
\begin{equation}\label{eqn:oscillation}
{\rm osc}^2_{\tau}(u_h) = h_{\tau}^4\left (\|\nabla
u_h\|_{0,\tau}^2+\|\nabla G\|_{0,\tau}^2\right ).
\end{equation}

%theorem 7.1
\begin{theorem}\label{theorem7.1}
Let $u \in H^1(\Omega)$ be a weak solution of the RPBE and $u_h$ be the
finite element approximation with a grid satisfying assumptions {\rm (A1)}
and {\rm (A2)}. There exist two constants depending only on the shape
regularity of $\mathcal T_h$ such that
\begin{equation} 
\| u - u_h \|_{1}^2 \le
C_1 \, \eta _h^2 + C_2 \, {\rm osc}_h^2,\label{eq7.6}
\end{equation}
where
$$
\eta _h^2:=\sum_{\tau \in \mathcal T_h} \eta_{\tau}^2(u_h), \quad
\text{and}\quad {\rm osc}_h^2:=\sum _{\tau \in \mathcal T_h\cap \Omega
 _s}{\rm osc}^2_{\tau}(u_h).
$$
\end{theorem}
\unskip

%\begin{proof}
\emph{Proof.} 
We shall apply the general estimate in 
\cite[Chapter 2]{Verfurth.R1996} (see also~\cite{Verfurth.R1994}) to 
$$
\underline{a}(x,u,\nabla u)=\varepsilon \nabla u+(\varepsilon
-\varepsilon _m)\nabla G, \quad \text{and}\quad b(x,u,\nabla
u)=-\bar{\kappa} ^2\sinh(u+G).
$$ 
We then use the following facts to get the desired result:
\begin{itemize}
\item $\nabla\cdot (\varepsilon\nabla u_h)~|_{\tau}=0\ \forall \tau
\in \mathcal T_h$ by the assumption (A2) of the grid;

\item $\nabla \cdot ((\varepsilon-\varepsilon _m)\nabla G)
 ~|_{\tau}=0\ \forall \tau \in \mathcal T_h$ since $\Delta G(x)=0$
 if $x\notin \{x_i\}$.

\item For $\tau \in \mathcal T_h\cap \Omega _s$, let $\bar u_h$ and 
$\bar G$ denote the average of $u_h$ and $G$ over $\tau$, 
respectively. We then have
\begin{eqnarray*}
\|\sinh (u_h+G)-\sinh (\bar u_h + \bar G)\|_{0,\tau}&\leq &|\cosh
(\xi)|\|u_h-\bar u_h + G- \bar G\|_{0,\tau}\\ &\leq
&Ch_{\tau}^2(\|\nabla u_h\|_{0,\tau}+\|\nabla G\|_{0,\tau}).
\end{eqnarray*}
Here we use the $L^{\infty}$-estimates of $u$ and $u_h$ to conclude
that $|\cosh (\xi)|\leq C$ and the standard error estimate for 
$\|u_h-\bar u_h\|_{0,\tau}$ and $\|G-\bar G\|_{0,\tau}$.\qquad\endproof
\end{itemize}
%\end{proof}

We give some remarks on our error estimator and the oscillation 
term. First, using (\ref{corcevity}) one can easily show that $\|\nabla
u_h\|_{0,\Omega}\leq C$ uniformly with respect to $h$ and thus ${\rm 
  osc} _{\tau}=O(h^2_{\tau})$. Comparing to the order of $\eta 
_{\tau}=O(h_{\tau})$, the error estimator $\eta _{\tau}$ will dominate 
in the upper bound. Second, in (\ref{eqn:estimator}) the jump of
$\left [ n_S \cdot (\varepsilon- \varepsilon _m)\nabla G\right]_S\neq 
0$ only if $S\in \Gamma _h$. This additional term with order 
$O([\varepsilon])$ will emphasize the elements around the interface
where the refinement most occurs.

Although it is clear that the upper bound is the key to bounding the 
error, the lower bound can also be quite useful; it can help to ensure 
that the adaptive procedure does not do too much work by overrefining
an area where it is unnecessary. Again using the general framework for 
the  a posteriori  error estimate in
\cite{Verfurth.R1994,Verfurth.R1996}, we have the following lower 
bound result.

%theorem 7.2
\begin{theorem}\label{theorem7.2}
There exists two constants $C_3, C_4$  depending only on the shape
regularity of $\mathcal T_h$ such that
$$
\eta _{\tau}^2(u_h)\leq C_3\|u-u_h\|_{1,\omega _{\tau}}^2+C_4\sum
_{\tilde \tau\in \omega _{\tau}\cap \Omega _s}{\rm osc}^2_{\tilde
 \tau}(u_h) \quad \forall \tau \in \mathcal T_h.
$$
\end{theorem}
\unskip

%sec 7.2
\subsection{Marking and refinement strategy}\label{sec7.2}
Given an initial triangulation $\mathcal T_0$, we shall generate a 
sequence of nested conforming triangulations $\mathcal T_k$ using the 
following loop:
\begin{equation}\label{keyloop}
\hbox {\bf \small SOLVE } \rightarrow \hbox{ \bf \small ESTIMATE}
\rightarrow \hbox{ \bf \small MARK } \rightarrow \hbox{ \bf \small
 REFINE}.
\end{equation}
More precisely to get $\mathcal T_{k+1}$ from $\mathcal T_k$ we first 
solve the discrete equation to get $u_k$ on $\mathcal T_k$. The error 
is estimated using $u_k$ and used to mark a set of triangles that 
are to be refined. Triangles are refined in such a way that the 
triangulation is still shape-regular and conforming.

We have discussed the step {\bf \small ESTIMATE} in detail, and we
shall not discuss the step {\bf \small SOLVE}, which deserves a
separate investigation. We assume that the solutions of the 
finite-dimensional problems can be solved to any accuracy 
efficiently. Examples of such optimal solvers are the multigrid method or the
multigrid-based preconditioned conjugate gradient method 
\cite{Xu.J1992a,Bramble.J;Zhang.X2000,Hackbusch.W1985,Xu.J;Zikatanov.L2002}. 
In particular we refer to 
\cite{Aksoylu.B;Bond.S;Holst.M2003,Aksoylu.B;Holst.M2006} for recent 
work on adaptive grids in three dimensions and
\cite{Holst.M;Said.F1995,Holst.M1994} for solving the PBE with inexact
Newton methods.

We now present the marking strategy which is crucial for our adaptive 
methods. We shall focus on one iteration of loop (\ref{keyloop}) and 
thus use $\mathcal T_H$ for the coarse mesh and $\mathcal T_h$ for the 
refined mesh. Quantities related to those meshes will be distinguished 
by a subscript $H$ or $h$, respectively.

Let $\theta _i, i=1,2$ be two numbers in $(0,1)$.
\begin{enumerate}
\item Mark $\mathcal M_{1,H}$ such that
$$
\sum _{\tau \in \mathcal M_{1,H}}\eta _{\tau}^2(u_H)\geq \theta
 _1\sum _{\tau \in \mathcal T_{H}}\eta _{\tau}^2(u_H).
$$
\item If
\begin{equation}\label{small}
{\rm osc}_H\geq \eta _H
\end{equation}
or
\begin{equation}\label{small2}
C_4\sum _{\tilde \tau \in \cup _{\tau \in M_H}\omega _{\tau}}{\rm
 osc}_{\tau}^2(u_H)\geq \frac{1}{2}\sum _{\tau \in \mathcal M_H}\eta _\tau^2(u_H),
\end{equation}
then extend $\mathcal M_{1,H}$ to $\mathcal M_{2,H}$ such that
$$
\sum _{\tau \in \mathcal M_{2,H}}{\rm osc} _{\tau}^2(u_H)\geq
\theta _2\sum _{\tau \in \mathcal T_{H}}{\rm osc}_{\tau}^2(u_H).
$$
\end{enumerate}
Unlike the marking strategy for reducing oscillation in the adaptive 
finite element methods in 
\cite{Morin.P;Nochetto.R;Siebert.K2000,Morin.P;Nochetto.R;Siebert.K2002}, 
in the second step, we put a switch (\ref{small})--(\ref{small2}). In 
our setting, the oscillation ${\rm osc}_H=O(H^2)$ is in general a high-order term. The marking step (2) is seldom applied.

In the {\bf \small REFINE} step, we need to carefully choose the rule 
for dividing the marked triangles such that the mesh obtained by this 
dividing rule is still conforming and shape-regular. Such refinement
rules include red and green refinement 
\cite{Bank.R;Sherman.A;Weiser.A1983}, longest refinement 
\cite{Rivara.M1984,Rivara.M1984a}, and newest vertex bisection
\cite{Sewell.E1972, Mitchell.W1988, Mitchell.W1989}. For the {\bf \small REFINE} step, we are going to
impose the following assumptions.

(A3) Each $\tau \in \mathcal M_H$, as well as each of
its faces, contains a node of $\mathcal T_h$ in its interior.

(A4) Let $\mathcal T_h$ be a refinement of $\mathcal T_H$ such
that the corresponding finite element spaces are nested,
i.e., $V^H\subset V^h$.

With those assumptions, we can have the discrete lower bound between 
two nested grids. Let $\mathcal T_H$ be a shape-regular triangulation,
and let $\mathcal T_h$ be a refinement of $\mathcal T_H$ obtained by 
local refinement of marked elements set $\mathcal M_H$. The assumption 
(A3) is known as the  {\it interior nodes property} in
\cite{Morin.P;Nochetto.R;Siebert.K2002}. Such a requirement ensures 
that the refined finite element space $V^h$ is fine enough to capture 
the difference of solutions.

%theorem 7.3
\begin{theorem}\label{th:dislower}
Let $\mathcal T_H$ be a shape-regular triangulation, and let
$\mathcal T_h$ be a refinement of $\mathcal T_H$ obtained by some 
local refinement methods of marked elements set $\mathcal M_H$, such 
that it satisfies assumptions {\rm (A3)} and {\rm (A4)}. Then there
exist two constants, depending only on the shape regularity of
$\mathcal T_H$, such that
\begin{equation}\label{lowerbd}
\eta _{\tau}^2(u_H)\leq C_3\|u_h-u_H\|^2_{1,\omega _{\tau}}+C_4 \sum
_{\tilde \tau \in \omega _{\tau}}{\rm osc}^2_{\tilde \tau}(u_H)\quad
\forall \tau \in \mathcal M_H.
\end{equation}
\end{theorem}
\unskip

\begin{proof}
The proof is standard using the discrete bubble functions on $\tau$ 
and each face $S\in \partial \tau$.%
\qquad\end{proof}

%7.3
\subsection{Convergence analysis}\label{sec7.3}
We shall prove that the repeating of loop (\ref{keyloop}) will produce a
convergent solution $u_k$ to $u$. The convergent analysis of the adaptive
finite element method is an active topic. In the literature it is mainly
restricted to the linear 
equations~\cite{Chen.L;Holst.M;Xu.J2006,Stevenson.R2005a,Carstensen.C;Hoppe.R2005a,Morin.P;Nochetto.R;Siebert.K2000,Dorfler.W1996,Binev.P;Dahmen.W;DeVore.R2004,Morin.P;Nochetto.R;Siebert.K2002,Mekchay.K;Nochetto.R2005,Dorfler.W;Wilderotter.O2000,Bansch.E;Morin.P;Nochetto.R2002}. 
The convergence analysis for the nonlinear equation is relatively 
rare~\cite{Dorfler.W1995,Veeser.A2002,Carstensen.C2006}.

%lemma 7.4
\begin{lemma}\label{lemma7.4}
Let $\mathcal T_H$ and $\mathcal T_h$ satisfy assumptions 
{\rm (A1)--(A4)}. Then there exist two constants  depending only on
the shape regularity of $\mathcal T_H$ such that
$$
\|u-u_H\|_1^2\leq C_5\|u_h-u_H\|_1^2+ C_6\, {\rm osc}_H^2.
$$
When {\rm (\ref{small})} and {\rm (\ref{small2})} do not hold, we 
have a stronger inequality
$$
\|u-u_H\|_1^2\leq C_7\, \|u_h-u_H\|_1^2,
$$
where $C_7$ depends only on the shape regularity of $\mathcal T_H$.
\end{lemma}

\begin{proof} 
By the upper bound and marking strategy
\begin{eqnarray*}
\|u-u_H\|_1^2 &\leq & C_1\eta _{H}^2 + C_2{\rm osc}_{H}^2\\ &\leq &
C_1\theta _1^{-1}\sum _{\tau\in \mathcal M_{1,H}}\eta _{\tau}^2(u_H) +
C_2\, {\rm osc}_{H}^2\\ &\leq & C_5\|u_h-u_H\|_1^2+C_6\, {\rm osc}_H^2,
\end{eqnarray*}
with
$$
C_5=C_1\theta _1^{-1}C_3^{-1}, \quad \text{ and }\quad
C_6=(C_2+2C_3^{-1}C_4).
$$
If (\ref{small}) does not hold, i.e., ${\rm osc}_{H}\leq \eta _H$,
the first inequality becomes
$$
\|u-u_H\|_1^2\leq (C_1+C_2)\eta _{H}^2.
$$
If (\ref{small2}) does not hold, we can easily modify the lower 
bound (\ref{lowerbd}) as
$$
\sum _{\tau \in \mathcal M_{1,H}}\eta _{\tau}^2(u_H)\leq
2C_3\|u_h-u_H\|_{1}^2.
$$
Then the inequality follows similarly.%
\qquad\end{proof}

For $\tau _h\subset \tau _H$, let $h_{\tau _h}=\gamma H_{\tau _H}$,
with $\gamma \in (0,1)$. The next lemma shows that even the 
oscillation is not small; there is also a reduction result. For the
marked set $\mathcal M_{H}\subset \mathcal T_H$, we shall use 
$\overline{\mathcal M_{H}}$ to denoted the refined elements in 
$\mathcal T_h$.

%lemma 7.5
\begin{lemma}\label{lemma7.5}
If $\mathcal M_{2,H}\backslash \mathcal M_{1,H}\notin \varnothing$, 
there exist $\rho _1,\rho _2$ such that
$$
{\rm osc}_h^2\leq \rho _1 \, {\rm osc}_H^2 + \rho_2\|u_h-u_H\|_1^2.
$$
\end{lemma}
\unskip

\begin{proof}
\begin{eqnarray*}
{\rm osc}_h^2 &\leq & \sum _{\tau\in \mathcal T_h}{\rm
 osc}_{\tau}^2(u_H) + C\sum _{\tau\in \mathcal
 T_h}(h_{\tau}^4\|\nabla (u_h-u_H)\|^2_{\tau})\\ &\leq & \sum
_{\tau_h \in \overline{\mathcal M}_{2,H}}{\rm osc}_{\tau}^2(u_H) +
\sum _{\tau _h\in \mathcal T_h\backslash \overline{\mathcal
  M}_{2,H}}{\rm osc}_{\tau}^2(u_H) + Ch^2\|\nabla
(u_h-u_H)\|^2\\ &\leq & \gamma ^2\sum _{\tau _H\in \mathcal
 M_{2,H}}{\rm osc}_{\tau}^2(u_H) + \sum _{\tau_H \in \mathcal
 T_H\backslash \mathcal M_{2,H}}{\rm osc}_{\tau}^2(u_H) +
Ch^2\|\nabla (u_h-u_H)\|^2\\ &\leq & {\rm osc}_H^2+(\gamma ^2-1)\sum
_{\tau _H\in \mathcal M_{2,H}}{\rm osc}_{\tau}^2(u_H) + Ch^2\|\nabla
(u_h-u_H)\|^2\\ &\leq & \rho _1 \, {\rm osc}_H^2 + \rho
_2\|u_h-u_H\|^2_1,
\end{eqnarray*}
with $\rho _1 =1-(1-\gamma^2)/\theta _2\in (0,1)$, and $\rho_2=Ch^2$.%
\qquad\end{proof}

We shall choose $\theta _2$ sufficiently close to $1$ and 
$h_{\max}<1/c$ to ensure $\rho _i\in (0,1),i=1,2$. 

For the nonlinear problem, we do not have the orthogonality in $H^1$ 
norms. But we shall use the trivial identity
\begin{equation}\label{energy}
E(u_H)-E(u)=E(u_H)-E(u_h)+E(u_h)-E(u).
\end{equation}
The following lemma proves the equivalence of energy error and error 
in $H^1$ norm. Again the $L^{\infty}$ norm estimate of $u$ and $u_h$ 
is crucial.

%lemma 7.6
\begin{lemma}\label{lm:energynorm}
If both $\mathcal T_h$ and $\mathcal T_H$ satisfy the assumption {\rm (A1)}, 
then
\begin{itemize}
\item $E(u_h)-E(u)\simeq \|u_h-u\|_1^2$;
\item $E(u_H)-E(u)\simeq \|u_H-u\|_1^2$;
\item $E(u_H)-E(u_h)\simeq \|u_H-u_h\|_1^2$.
\end{itemize}
\end{lemma}

\begin{proof}
By the Taylor expansion
$$
E(u_H)-E(u_h)=\langle DE(u_h), u_H-u_h \rangle
+(D^2E(\xi)(u_H-u_h), u_H-u_h).
$$
The first term is zero since $u_h$ is the minimizer. The desired 
result follows from the bound
$$
{\bar \kappa} ^2\leq \|D^2E(\xi)\|_{\infty}={\bar \kappa}^2
\|\cosh(\xi +G)\|_{\infty,\Omega _s}\leq C.
$$
Other inequalities follow from the same line.%
\qquad\end{proof}

Our adaptive finite element methods (AFEMs) consist of the iteration
of loop (\ref{keyloop}) with the estimate, marking, and refinement
parts discussed before. Also the grids generated by the algorithm will
satisfy assumptions (A1)--(A4). Hereafter we replace the subscript $h$ 
by an iteration counter called $k$ and introduce some notation to 
simplify the proof. Let $u_k$ be the solution in the $k$th iteration,
$\delta _k:=E(u_k)-E(u)$, $d_k=E(u_k)-E(u_{k+1})$, and $o_k={\rm osc}^2(u_k)$

%theorem 7.7
\begin{theorem}\label{theorem7.7}
The adaptive method using loop {\rm (\ref{keyloop})} will produce a 
convergent approximation in the sense that
$$
\lim _{k\rightarrow 0}\|u-u_k\|_1=0.
$$
\end{theorem}
\unskip

\begin{proof}
By Lemma \ref{lm:energynorm}, we need only to show $\delta^k\rightarrow 0$
as $k\rightarrow 0$. We first discuss the easier case: When ${\rm osc}_H$ 
is the high-order term in the sense that the inequalities (\ref{small})
and (\ref{small2}) do not hold, we have the error reduction
$$
\|u-u_H\|^2_1\leq C\|u_h-u_H\|^2.
$$
Using Lemma \ref{lemma7.5} and (\ref{energy}), we have
$$
E(u_H)-E(u)\leq C(E(u_H)-E(u_h)),
$$
which is equivalent to $\delta _H\leq C\delta _H-C\delta _h.$ 
Then $\delta _h\leq (1-1/C)\delta _H,$ and thus
$$
\delta ^k\leq \alpha ^k \delta ^0, \quad \text{with}\; \alpha=(1-1/C)\in (0,1).
$$
When the oscillation is not small, i.e., (\ref{small}) or
(\ref{small2}) holds, we can get only
\begin{equation}\label{eqn:1}
\Lambda _1\delta _{k}\leq d_k+\Lambda _2 o_k, \quad \text{with}\;
\Lambda _1\in (0,1).
\end{equation}
We shall use techniques from~\cite{Mekchay.K;Nochetto.R2005} to prove 
the convergence. Recall that we have
\begin{equation}\label{eqn:2}
\delta _{k+1}=\delta _{k}-d_k.
\end{equation}
For any $\beta \in (0,1)$, $\beta \times (\ref{eqn:1})+(\ref{eqn:2})$ 
gives
\begin{equation}\label{eqn:3}
\delta _{k+1}\leq \alpha \delta _k +\beta \Lambda _2 o_k-(1-\beta)d_k,
\quad \text{with}\; \alpha = (1-\beta \Lambda _1)\in (0,1).
\end{equation}
Recall that we have
\begin{equation}\label{eqn:4}
o_{k+1}\leq \rho _1o_k+\rho _2 d_k.
\end{equation}
Let $\gamma = (1-\beta)/\rho _2$; $(\ref{eqn:4})\times \gamma +
(\ref{eqn:3})$ gives\vspace{-3pt}
$$
\delta _{k+1}+\gamma o_{k+1}\leq \alpha \delta _{k}+ (\beta \Lambda_2 
+ \rho _1\gamma)o_k.
$$
Let $1>\mu > \rho _1$. We choose\vspace{-3pt}
$$
\beta = \frac{\frac{\mu - \rho _1}{\rho _2}}{\Lambda _2+\frac{\mu
 -\rho _1}{\rho _2}}\in (0,1)
$$
to get\vspace{-3pt}
$$
\delta _{k+1}+\gamma o_{k+1}\leq \max (\alpha, \mu) (\delta_k+\gamma o_k),
$$
which also implies the convergence of our AFEM.%
\qquad\end{proof}\enlargethispage{-6pt}

%sec 8
\section{Summary and concluding remarks}\label{sec8}
In this article we have established a number of basic theoretical 
results for the nonlinear Poisson--Boltzmann equation and for its 
approximation using finite element methods. We began by showing that 
the problem is well-posed through the use of an auxiliary or {\em 
regularized} version of the equation and then established a number
of basic estimates for the solution to the regularized problem. The 
Poisson--Boltzmann equation does not appear to have been previously 
studied in detail theoretically, and it is hoped that this paper
will help provide 
molecular modelers with a better theoretical foundation for their 
analytical and computational work with the Poisson--Boltzmann equation. 
The bulk of this article then focused on designing a numerical 
discretization procedure based on the regularized problem and on
establishing rigorously that the discretization procedure converged to 
the solution to the original (nonregularized) nonlinear
Poisson--Boltzmann equation. Based on these results, we also designed 
an adaptive finite element approximation procedure and then gave a
fairly involved technical argument showing that this adaptive 
procedure also converges in the limit of mesh refinement. This 
article apparently gives the first convergence result for a numerical 
discretization technique for the nonlinear Poisson--Boltzmann equation 
with delta distribution sources, and it also introduces the first provably
convergent adaptive method for the equation. This last result is one 
of only a handful of convergence results of this type for nonlinear 
elliptic equations (the others being 
\cite{Dorfler.W1995,Veeser.A2002,Carstensen.C2006}). 

\looseness=-1Several of the theoretical results in the paper rest on some basic
assumptions on the underlying simplex mesh partitioning of the domain, 
namely, assumptions (A1)--(A4); we now make a few comments on these
assumptions. To begin, we required a refinement procedure that would 
preserve the $L^{\infty}$ norm estimate of $u_h$. Meeting this
requirement in the two-dimensional setting is relatively easy; one can 
choose $\Omega$ as a square and start with a uniform mesh of a 
square. For the refinement methods, one can use longest edge or newest 
vertex bisection. Subdivisions obtained by these two methods contain 
only one type of triangle: isosceles right triangles. Thus the
assumption (A1$'$) always holds. In the three-dimensional setting, this
is more tricky. Bisection will introduce some obtuse angles in the 
refined elements. One needs to use a three-dimensional analogue of 
red-green 
refinement~\cite{Bornemann.F;Erdmann.B;Kornhuber.R1993}. However, this 
will not produce nested subspaces; i.e., assumption (A4) is 
invalid. For convergence analysis based on red-green refinement, we 
could use the technique in~\cite{Stevenson.R2005b} to relax the 
assumption (A4). Since this will only add technical difficulties but 
does not exhibit principally new phenomena, we omit them here. Another
approach to relax the assumption (A1) is to use pointwise a
posteriori error estimates developed
in~\cite{Nochetto.R;Schmidt.A;Siebert.K;Veeser.A2006} for monotone 
semilinear equations. We can start with a quasi-uniform triangulation
and refine the triangulation according to the pointwise a
posteriori error estimator to make sure $\|u-u_h\|_{\infty}\leq C$.
Then together with the $L^{\infty}$ norm estimate of $u$, by the
triangulation inequality $\|u_h\|_{\infty}\leq 
\|u\|_{\infty}+\|u-u_h\|_{\infty}\leq C$, we have the control of 
$\|u_h\|_{\infty}$. Note that the pointwise a posteriori
error estimates developed 
in~\cite{Nochetto.R;Schmidt.A;Siebert.K;Veeser.A2006} are for elliptic-type equations with continuous coefficients. To use this approach we
need to adapt the estimate for the jump coefficients case which will 
be a further research topic. 

\looseness=-1Assumption (A2) is needed to approximate the interface well in an
a priori manner. Of course, one can include this approximation
effect into the a posteriori error estimate (namely, the term
$\|F(u_h)-F_h(u_h)\|$) and use this to drive local refinement to 
improve the approximation to the desired level for the assumption or 
use the strategy for the oscillation to include it in the refinement 
loop. However, we note that, since the interface is known  a
priori from, e.g., x-ray crystallography information, we do not need
to solve the equation (which is generally the more expensive route) to 
solve this problem; we view this as primarily a mesh generation 
problem. Robust algorithms to produce well-shaped tetrahedral meshes 
which are constrained to exactly match some interior embedded 
two-manifold are available in the literature; for example, see
\cite{Chen.L;Holst.M2006,Alliez.P;Cohen-Steiner.D;Yvinec.M;Desbrun.M2005a}. 
A simple algorithm can be based entirely on local refinement with the 
marking and refinement strategy, but without having to solve the PBE to 
produce error indicators: If the element cross the interface, then it
gets refined. This strategy was employed in 
~\cite{Baker.N;Holst.M;Wang.F2000}.

After this work was done, we learned that the assumption (A3) is
not needed for the convergence of adaptive finite element methods for a
linear elliptic equation. As an ongoing project, we are extending it 
to the nonlinear Poisson--Boltzmann equation.

Finally, we make some remarks on the practical realization of a 
convergent discretization procedure based on the two-way (or 
three-way) expansion into a known singular function and solution(s) 
of an associated regularized version of the problem. Methods for 
building high-quality approximate solutions of the regularized nonlinear 
PBE, either by solving (\ref{RPBE1})--(\ref{RPBE2}) at once or by
solving for the linear and nonlinear pieces separately by solving 
(\ref{linear})--(\ref{nonlinear}) and then adding the solutions 
together, are well-understood. The techniques described in 
~\cite{Holst.M;Baker.N;Wang.F2000}, taken together with the 
approximation framework and the adaptive algorithm proposed in the 
present article, moves us a step closer to the goal of a complete 
optimal solution to this problem, in terms of approximation quality 
for a given number of degrees of freedom, computational complexity of 
solving the corresponding discrete equations, and the storage 
requirements of the resulting algorithms. What remains is simply the 
cost of evaluating the singular function $G$ in forming the source 
terms in (\ref{RPBE1}) or (\ref{linear}). The source terms are 
evaluated using numerical quadrature schemes: sampling the integrand 
at specially chosen discrete points in each element and then summing
the results up using an appropriate weighting. This is equivalent to 
computing all pairwise interactions between the collection of 
quadrature points (a fixed constant number of points per simplex) and 
the number of charges forming $G$. Given that $G$ is typically formed 
from at most a few thousand charges, the algorithm evaluating $G$ at 
the quadrature points should scale linearly with the number of 
quadrature points, which is a (small) constant multiple of the number 
of simplices. This can be accomplished using techniques such 
distance-classing and fast multiple-type methods.

\bibliographystyle{abbrv}
\bibliography{../bib/books,../bib/papers,../bib/mjh,../bib/library,../bib/ref-gn,../bib/coupling,../bib/pnp}

%\clearpage
%\input{app}

\vspace*{0.5cm}

\end{document}